\documentclass[11pt]{scrartcl}

\usepackage[english]{babel}
\usepackage[utf8]{inputenc}		

\usepackage{url}

\newcommand{\union}{\ensuremath{\cup}}		
\newcommand{\isdef}{\ensuremath{\mathrel{\mathop:}=}}		

\usepackage{enumitem}
\usepackage{amsfonts, amsmath, amssymb}
\usepackage{amsthm}			

\usepackage{pgf, tikz}
\usepackage{tikz-cd}
\usetikzlibrary{calc}
\usetikzlibrary{arrows, arrows.meta}
\usetikzlibrary{positioning}

\usepackage{wrapfig}
\usepackage{nicefrac}					
\usepackage{color}
\newcommand{\norm}[3][black]{\ensuremath{\left\Vert #3\right\Vert_{\textcolor{#1}{#2}}}}

\newcommand{\bisn}[1]{\ensuremath{\{1,\dots , #1 \}}}	

\newcommand{\ohne}{\ensuremath{\backslash}}

\DeclareMathOperator{\Sym}{Sym}
\DeclareMathOperator{\Pot}{Pot}

\DeclareMathOperator{\dist}{dist}

\newcommand{\menge}[1]{\ensuremath{\mathbb{#1}}}
\newcommand{\N}{\menge{N}}
\newcommand{\Z}{\menge{Z}}

\newcommand{\C}{\menge{C}}

\newcommand{\e}{\text{e}}				

\renewcommand{\phi}{\varphi}
\renewcommand{\epsilon}{\varepsilon}
\renewcommand{\subset}{\subseteq}

\newtheorem{lem}{Lemma}[section]		
\newtheorem{ex}[lem]{Example}		
\newtheorem{rem}[lem]{Remark}
\newtheorem{theo}[lem]{Theorem}			
\newtheorem{cor}[lem]{Corollary}		
\newtheorem{defi}[lem]{Definition}	
  {\begin{proof}[Well--defined]}
  {\end{proof}}

\newcommand{\tikzAngleOfLine}{\tikz@AngleOfLine}
\def\tikz@AngleOfLine(#1)(#2)#3{%
  \pgfmathanglebetweenpoints{%
    \pgfpointanchor{#1}{center}}{%
    \pgfpointanchor{#2}{center}}
  \pgfmathsetmacro{#3}{\pgfmathresult}%
}

\newcommand{\triangleMarkedLine}[4][black]{
	\draw[#1] (#2) -- (#3);
	\tikzAngleOfLine(#2)(#3){\angle}
	\def\dir{#4}
	\def\dist{0.18}
	\ifnum\dir>0
		\draw[#1,-open triangle 60] ($(#2)!0.5!(#3)$) -- +(\angle-90:\dist);
	\else
		\draw[#1,-open triangle 60] ($(#2)!0.5!(#3)$) -- +(\angle+90:\dist);
	\fi
}

\usepackage{amsmath, amssymb, amsfonts} 

\usepackage{ifthen}

\usepackage{tikz}
\usetikzlibrary{calc}
\usetikzlibrary{positioning}
\usetikzlibrary{shapes}
\usetikzlibrary{patterns}

\usepackage{tikz-cd}

\makeatletter
\edef\texforht{TT\noexpand\fi
  \@ifpackageloaded{tex4ht}
    {\noexpand\iftrue}
    {\noexpand\iffalse}}
\makeatother

\makeatletter
\newif\iftikz@node@phantom
\tikzset{
  phantom/.is if=tikz@node@phantom,
  text/.code=%
    \edef\tikz@temp{#1}%
    \ifx\tikz@temp\tikz@nonetext
      \tikz@node@phantomtrue
    \else
      \tikz@node@phantomfalse
      \let\tikz@textcolor\tikz@temp
    \fi
}
\usepackage{etoolbox}
\patchcmd\tikz@fig@continue{\tikz@node@transformations}{%
  \iftikz@node@phantom
    \setbox\pgfnodeparttextbox\hbox{}
  \fi\tikz@node@transformations}{}{}
\makeatother

\tikzset{ 
    vertexNodePlain/.style = {fill=#1, shape=circle, inner sep=0pt, minimum size=2pt, text=none},
    vertexNodePlain/.default=gray,
    vertexPlain/labels/.style = {
        vertexNode/.style={vertexNodePlain=##1},
        vertexLabel/.style={gray}
    },
    vertexPlain/nolabels/.style = {
        vertexNode/.style={vertexNodePlain=##1},
        vertexLabel/.style={text=none}
    },
    vertexPlain/.style = vertexPlain/#1,
    vertexPlain/.default=labels
}
\tikzset{
    vertexNodeNormal/.style = {fill=#1, shape=circle, inner sep=0pt, minimum size=4pt, text=none},
    vertexNodeNormal/.default = blue,
    vertexNormal/labels/.style = {
        vertexNode/.style={vertexNodeNormal=##1},
        vertexLabel/.style={blue}
    },
    vertexNormal/nolabels/.style = {
        vertexNode/.style={vertexNodeNormal=##1},
        vertexLabel/.style={text=none}
    },
    vertexNormal/.style = vertexNormal/#1,
    vertexNormal/.default=labels
}
\tikzset{
    vertexNodeBallShading/pdf/.style = {ball color=#1},
    vertexNodeBallShading/svg/.style = {fill=#1},
    vertexNodeBallShading/.code = {
        \if\texforht
            \tikzset{vertexNodeBallShading/svg=#1!90!black}
        \else
            \tikzset{vertexNodeBallShading/pdf=#1}
        \fi
    },
    vertexNodeBall/.style = {shape=circle, vertexNodeBallShading=#1, inner sep=2pt, outer sep=0pt, minimum size=3pt, font=\tiny},
    vertexNodeBall/.default = orange,
    vertexBall/labels/.style = {
        vertexNode/.style={vertexNodeBall=##1, text=black},
        vertexLabel/.style={text=none}
    },
    vertexBall/nolabels/.style = {
        vertexNode/.style={vertexNodeBall=##1, text=none},
        vertexLabel/.style={text=none}
    },
    vertexBall/.style = vertexBall/#1,
    vertexBall/.default=labels
}
\tikzset{ 
    vertexStyle/.style={vertexNormal=#1},
    vertexStyle/.default = labels
}

\newcommand{\vertexLabelR}[4][]{
    \ifthenelse{ \equal{#1}{} }
        { \node[vertexNode] at (#2) {#4}; }
        { \node[vertexNode=#1] at (#2) {#4}; }
    \node[vertexLabel, #3] at (#2) {#4};
}
\newcommand{\vertexLabelA}[4][]{
    \ifthenelse{ \equal{#1}{} }
        { \node[vertexNode] at (#2) {#4}; }
        { \node[vertexNode=#1] at (#2) {#4}; }
    \node[vertexLabel] at (#3) {#4};
}

\newcommand{\edgeLabelColor}{blue!20!white}
\tikzset{
    edgeLineNone/.style = {draw=none},
    edgeLineNone/.default=black,
    edgeNone/labels/.style = {
        edge/.style = {edgeLineNone=##1},
        edgeLabel/.style = {fill=\edgeLabelColor,font=\small}
    },
    edgeNone/nolabels/.style = {
        edge/.style = {edgeLineNone=##1},
        edgeLabel/.style = {text=none}
    },
    edgeNone/.style = edgeNone/#1,
    edgeNone/.default = labels
}
\tikzset{
    edgeLinePlain/.style={line join=round, draw=#1},
    edgeLinePlain/.default=black,
    edgePlain/labels/.style = {
        edge/.style={edgeLinePlain=##1},
        edgeLabel/.style={fill=\edgeLabelColor,font=\small}
    },
    edgePlain/nolabels/.style = {
        edge/.style={edgeLinePlain=##1},
        edgeLabel/.style={text=none}
    },
    edgePlain/.style = edgePlain/#1,
    edgePlain/.default = labels
}
\tikzset{
    edgeLineDouble/.style = {very thin, double=#1, double distance=.8pt, line join=round},
    edgeLineDouble/.default=gray!90!white,
    edgeDouble/labels/.style = {
        edge/.style = {edgeLineDouble=##1},
        edgeLabel/.style = {fill=\edgeLabelColor,font=\small}
    },
    edgeDouble/nolabels/.style = {
        edge/.style = {edgeLineDouble=##1},
        edgeLabel/.style = {text=none}
    },
    edgeDouble/.style = edgeDouble/#1,
    edgeDouble/.default = labels
}
\tikzset{
    edgeStyle/.style = {edgePlain=#1},
    edgeStyle/.default = labels
}

\newcommand{\faceColorY}{yellow!60!white}   
\newcommand{\faceColorB}{blue!60!white}     
\newcommand{\faceColorC}{cyan!60}           
\newcommand{\faceColorR}{red!60!white}      
\newcommand{\faceColorG}{green!60!white}    
\newcommand{\faceColorO}{orange!50!yellow!70!white} 

\newcommand{\faceColor}{\faceColorY}
\newcommand{\faceColorSwap}{\faceColorC}

\tikzset{
    face/.style = {fill=#1},
    face/.default = \faceColor,
    faceY/.style = {face=\faceColorY},
    faceB/.style = {face=\faceColorB},
    faceC/.style = {face=\faceColorC},
    faceR/.style = {face=\faceColorR},
    faceG/.style = {face=\faceColorG},
    faceO/.style = {face=\faceColorO}
}
\tikzset{
    faceStyle/labels/.style = {
        faceLabel/.style = {}
    },
    faceStyle/nolabels/.style = {
        faceLabel/.style = {text=none}
    },
    faceStyle/.style = faceStyle/#1,
    faceStyle/.default = labels
}
\tikzset{ face/.style={fill=#1} }
\tikzset{ faceSwap/.code=
    \ifdefined\swapColors
        \tikzset{face=\faceColorSwap}
    \else
        \tikzset{face=\faceColor}
    \fi
}

\usepackage{hyperref}

\newcommand{\colE}{blue}

\renewcommand{\bisn}[1]{\ensuremath{\underline{#1}}}	

\author{Markus Baumeister\footnote{Lehr- und Forschungsgebiet Algebra, RWTH Aachen University}}
\title{Foldability of simplicial surfaces onto a triangle}

\begin{document}
\maketitle

In recent years, complicated folding structures have been employed more
often, with applications ranging from architecture 
(\cite{BuriWeinand_OrigamiFoldedPlateStructures}) to robotics 
(\cite{Shigemune_OrigamiRobotPrinting}). Concurrently, mathematical aspects
of paper folding have been developed, 
spanning a diverse range of
topics (contrast \cite{Lang_OrigamiGeometricConstructions}, 
\cite{Lang_Origami4}, \cite{Hull94onthe}, 
\cite{Izmestiev_ClassificationKokotsakisPolyhedra},
\cite{SauerGraf_FlaechenverbiegungFacettenflaeche},
\cite{Stachel_KinematicKokotsakis},
\cite{Tachi_GeneralizationMeshOrigami}, \cite{WuYou_OrigamiQuaternions},
and \cite{DemRou}).
One interesting question is how a given surface built up from triangles
can be folded. There has been some progress towards answering this question, 
mostly with
explicit angles (\cite{BernHayes_ComplexityOfFlatOrigami}, \cite{Arkin2001}).
Since this problem is very hard in principle, we try to simplify it by
considering a purely combinatorial model of folding, independent of
any angles. 

In this paper, we analyse which surfaces can be folded
onto a triangle in this combinatorial model.
Considering only combinatorial folding alters the situation
drastically. For example, the octahedron can combinatorially be folded
onto a triangle, although it is rigid under regular folding 
(\cite{Dehn_StarrheitKonvexerPolyeder}). In contrast, a tetrahedron
cannot combinatorially be folded onto a triangle. In particular, the
notion of combinatorial folding focuses on intrinsic properties of
the surface instead of the embedding.

In order to describe combinatorial 
folding properly, we have to choose an appropriate
model. Although our model could be described as a triangulation of
a two--dimensional manifold (\cite{GrossYellenGraphs}, 
\cite{DeLoeraRambauSantosTriangulations}) or as a combinatorial manifold
(\cite{IzmestievKleeNovik_SimplicialMovesBalancedComplexes}),
we prefer a more intuitive,
combinatorial description,
which focuses on the incidence relations between vertices, edges, and faces
directly. We eschew the additional manifold structure since our analysis 
does not depend on it -- only the incidence relations 
between vertices, edges,
and faces are strictly necessary.

We will not describe the complete combinatorial folding theory in this paper,
and will restrict our definition in Section \ref{Sec_FoldingRestrictions}
to the specific case of ``folding onto a triangle''. Our main result is
that any combinatorial simplicial surface that can be folded
onto a triangle is orientable and admits a vertex--3--colouring. With
these properties given, we can reformulate the triangle--folding--problem
as the search for a cyclic permutation whose products with certain
involutions have a specified number of cycles (compare Corollary
\ref{Cor_FoldingOntoTriangle}
for the precise formulation).

\section{Covering a Triangle}\label{Sec_TriangleCovering}
This paper is concerned with the question 
``When can a simplicial surface be folded onto a single triangle?''.

Our definition of simplicial surfaces differs from that in the literature.
It is important to note that a simplicial surface (in our sense) is
not a simplicial complex, since different simplices may have the
same vertices. The closest definition can be found in \cite{Polthier2002}.

\begin{defi}\label{Def_SimplicialSurface}
    A \textbf{simplicial surface} is a quadruple $(V,E,F,\prec)$ such that
    \begin{enumerate}
        \item $V,E,F$ are finite sets (called \textbf{vertices}, \textbf{edges}, 
            and \textbf{faces}) and 
            $\prec\; \subset (V \times E) \uplus (V \times F) \uplus (E \times F)$
            is a transitive relation, called \textbf{incidence}.
        \item For every edge $e \in E$, there are exactly two vertices
            incident to $e$.
        \item For every face $f \in F$ there are three incident vertices
            $v_1,v_2,v_3$ and three incident edges $e_1,e_2,e_3$
            such that $v_i$ and $v_{i+1}$ are incident to
            $e_i$ for $1 \leq i \leq 3$ (where $v_4 \isdef v_1$).
        \item For every edge $e \in E$, there are at most two faces incident
            to $e$.
        \item For every vertex $v \in V$ there is a finite sequence
            $(e_1,f_1,e_2,f_2,\dots)$ such that
            \begin{itemize}
                \item The $e_i$ are pairwise distinct and exactly those edges 
                    incident to $v$.
                \item The $f_i$ are pairwise distinct and
                    exactly those faces incident to $v$.
                \item $e_i$ and $e_{i+1}$ are incident to $f_i$.
                \item If the final element of the sequence is a face, then $e_1$
                    is incident to that face.
            \end{itemize}
        \item For every vertex $v \in V$, there is an edge $e \in E$ such
            that $v \prec e$.
        \item For every edge $e \in E$, there is a face $f \in F$ with 
            $e \prec f$.
    \end{enumerate}
    The simplicial surface is \textbf{closed} if there are exactly two faces
    incident to every edge.
\end{defi}

While we allow our simplicial surfaces to be disconnected, it is often 
convenient to restrict to connected simplicial surfaces.

\begin{ex}\label{Ex_Triangle}
    For any set $M$ with three elements we can define a 
    \textbf{triangle} as the simplicial surface
    $(\Pot_1(M),\Pot_2(M),\Pot_3(M),\subsetneq)$.
\end{ex}

\begin{ex}\label{Ex_TorusIncidence}
    We can depict the incidence structure of a simplicial surface graphically.
    The following surface is a torus.
    
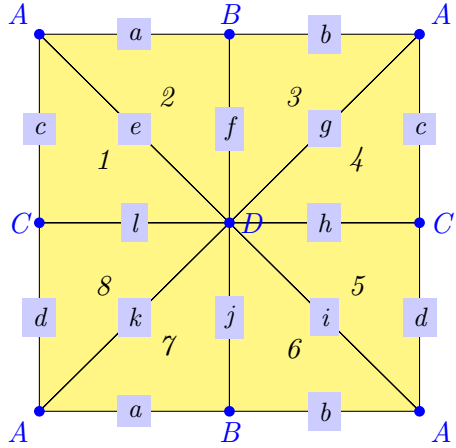
\begin{center}
        \begin{tikzpicture}[vertexStyle, edgeStyle, faceStyle]
            \def\len{2.5}
            \foreach \i in {0,1,2}{
                \foreach \j in {0,1,2}{
                    \coordinate (A\i\j) at (\len*\i,\len*\j);
                }
            }

            \draw[edge, face]
                (A02) -- node[edgeLabel]{c} (A01) -- (A11) -- cycle
                (A12) -- node[edgeLabel]{a} (A02) -- node[edgeLabel]{e} (A11) -- cycle
                (A22) -- node[edgeLabel]{b} (A12) -- node[edgeLabel]{f} (A11) -- cycle
                (A21) -- node[edgeLabel]{c} (A22) -- node[edgeLabel]{g} (A11) -- cycle
                (A20) -- node[edgeLabel]{d} (A21) -- node[edgeLabel]{h} (A11) -- cycle
                (A10) -- node[edgeLabel]{b} (A20) -- node[edgeLabel]{i} (A11) -- cycle
                (A00) -- node[edgeLabel]{a} (A10) -- node[edgeLabel]{j} (A11) -- cycle
                (A01) -- node[edgeLabel]{d} (A00) -- node[edgeLabel]{k} (A11) -- node[edgeLabel]{l} cycle;

            \foreach \p/\q/\n in {01/02/1, 02/12/2, 12/22/3, 22/21/4, 21/20/5, 20/10/6, 10/00/7, 00/01/8}{
                \node at (barycentric cs:A\p=1,A\q=1,A11=1) {\n};
            }

            \foreach \p/\d/\n in {00/below left/A, 10/below/B, 20/below right/A, 01/left/C, 11/right/D, 21/right/C, 02/above left/A, 12/above/B, 22/above right/A}{
                \vertexLabelR{A\p}{\d}{\n}
            }

        \end{tikzpicture}
    \end{center}
    Here, we have $V = \{A,B,C,D\}$, $E = \{a,b,\dots, l\}$ and 
    $F = \{1,\dots,8\}$.
\end{ex}

The process of ``folding onto a triangle'' 
can be separated into two conditions that need to be fulfilled.
The first one is a surjective map from the surface
to the triangle, and the second one consists of restrictions imposed by 
folding. We start with the
characterisation of the surfaces that can be mapped to a triangle.

\begin{defi}\label{Def_SimplicialMap}
    Let $S_1 = (V_1,E_1,F_1,\prec_1)$ and $S_2 = (V_2,E_2,F_2,\prec_2)$ be two
    simplicial surfaces. A \textbf{simplicial map} $\phi$ is a map
    \begin{equation*}
        \phi: V_1 \uplus E_1 \uplus F_1 \to V_2 \uplus E_2 \uplus F_2,
    \end{equation*}
    such that $\phi(V_1) \subset V_2$, $\phi(E_1) \subset E_2$, 
    $\phi(F_1) \subset F_2$ and $x \prec y$ implies $\phi(x) \prec \phi(y)$.

    If there is a simplicial map that is an inverse of $\phi$, then $\phi$ is called
    \textbf{simplicial isomorphism}.
\end{defi}

It is important to note that the vertices of a triangle can be coloured
with three distinct colours. The same colouring gives necessary and sufficient 
conditions for the existence of a map from the surface to the triangle.

\begin{defi}\label{Def_VertexColouring}
    Let $S = (V,E,F,\prec)$ be a simplicial surface. A 
    map $c_V: V \to \{1,2,3\}$ is called
    \textbf{vertex--3--colouring} if,
    for every edge $e \in E$ and distinct vertices $v_1,v_2 \in V$
    with $v_1 \prec e$ and $v_2 \prec e$, we have $c_V(v_1) \neq c_V(v_2)$.
\end{defi}

\begin{rem}\label{Rem_TriangleCoveringClassification}
    A simplicial surface admits a vertex--3--colouring if and only if
    there exists a simplicial map from the surface to the triangle.
\end{rem}

\section{Folding restrictions}\label{Sec_FoldingRestrictions}
After classifying all maps from a simplicial surface onto a 
triangle,
we incorporate the folding restrictions. To do so, we utilise
the edge--colouring which is induced by the 
vertex--colouring\footnote{This colouring can be interpreted
as proper colouring of the face--edge--graph of the simplicial surface.}.
\begin{defi}\label{Def_EdgeColouring}
    Let $S = (V,E,F,\prec)$ be a simplicial surface. A map 
    $c_E: E \to \{1,2,3\}$ is called
    \textbf{edge--3--colouring} if,
    for every face $f \in F$ and distinct edges $e_1,e_2 \in E$ with
    $e_1 \prec f$ and $e_2 \prec f$, we have $c_E(e_1) \neq c_E(e_2)$.

    Every vertex--3--colouring $c_V$ of $S$ defines an edge--3--colouring
    of $S$ via 
    $c_E(e) = c_V(v_1) + c_V(v_2) - 2$, where $v_1,v_2$ are 
    the vertices incident to $e$. This is called the 
    \textbf{induced edge--3--colouring}.
\end{defi}

We note in passing that our induced edge--3--colourings coincide with
the mmm--structures considered in \cite{PSNBsimplicialMMM}.

To formulate the folding restrictions, we need to be more specific about
the concept ``folding onto a triangle''. Here, we restrict our attention
to closed simplicial surfaces.

Intuitively, we want to formalise this picture:
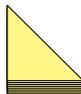
\begin{center}
    \begin{tikzpicture}[vertexStyle, edgeStyle, faceStyle]
        \coordinate (D1) at (0,0);
        \coordinate (D2) at (-1,1);
        \coordinate (D3) at (-1,0);
        \def\e{0.025}
        \foreach \i in {8,7,6,5,4,3,2,1}{
            \draw[edge, face] ($(D1)+(0,-\i*\e)$) -- ($(D2)+(0,-\i*\e)$) -- ($(D3)+(0,-\i*\e)$) -- cycle;
        }
    \end{tikzpicture}
\end{center}
Clearly, the faces are ordered by the folding. We model this
as a linear order on the faces of the simplicial surface.

Additionally, real materials usually do not self--intersect. As those
intersections can only become relevant at the edges, we focus our attention
there. Clearly, if edges of the simplicial surface are mapped to different
edges of the triangle, they do not come into conflict with each other. 
Therefore, we only
have to consider whether two edges that lie
in the same colour--class of the induced edge--3--colouring, are 
intersection--free.

Since we only consider closed simplicial surfaces,
each of these edges is adjacent to exactly two faces.
Since they come from an edge--3--colouring, these are four
different faces.
If these four faces are ordered
by the folding, one of the following cases has to manifest (we depict
faces as vertical lines and edges as half--circles):
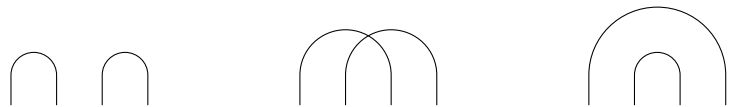
\begin{center}
    \newcommand{\faceLines}{
        \coordinate (A1) at (0,0);
        \coordinate (B1) at (0,-\len);
        \coordinate (A2) at (2*\inDist,0);
        \coordinate (B2) at (2*\inDist,-\len);
        \coordinate (A3) at (4*\inDist,0);
        \coordinate (B3) at (4*\inDist,-\len);
        \coordinate (A4) at (6*\inDist,0);
        \coordinate (B4) at (6*\inDist,-\len);

        \foreach \i in {1,2,3,4}{
            \draw (A\i) -- (B\i);
        }
    }
    \begin{tikzpicture}
        \def\inDist{0.3}
        \def\outDist{2}
        \def\len{0.4}

        \begin{scope}
            \faceLines
            \draw (A2) arc(0:180:\inDist);
            \draw (A4) arc(0:180:\inDist);
        \end{scope}
        \begin{scope}[shift={(6*\inDist+\outDist,0)}]
            \faceLines
            \draw (A3) arc(0:180:2*\inDist);
            \draw (A4) arc(0:180:2*\inDist);
        \end{scope}
        \begin{scope}[shift={(12*\inDist+2*\outDist,0)}]
            \faceLines
            \draw (A3) arc(0:180:\inDist);
            \draw (A4) arc(0:180:3*\inDist);
        \end{scope}
    \end{tikzpicture}
\end{center}
The edges only intersect in the middle case.
Therefore, we can define folding as follows:
\begin{defi}\label{Def_TriangleFolding}
    Let $(V,E,F,\prec)$ be a closed simplicial surface with 
    vertex--3--colouring $c_V$ and induced edge--3--colouring $c_E$. 
    A \textbf{triangle--folding} is a
    total ordering $<$ on $F$ such that:

    If there are two edges with faces $\{f_1,f_2\}$ and $\{g_1,g_2\}$,
    such that $f_1 < g_1 < f_2 < g_2$ holds, then these edges have
    different edge colours.
\end{defi}

\section{Relation between linear and cyclic orders}\label{Sec_Orders}
Folding onto a triangle requires a linear
order to be intersection--free. 
Since it is more convenient to work with a cyclic order,
we give a proof of their equivalence.

\begin{defi}\label{Def_CyclicOrder}
    Let $M$ be a finite set. A \textbf{cyclic order on $M$} is a
    cycle $\sigma \in \Sym(M)$ of length $|M|$.
\end{defi}

From a linear order we can construct a cyclic order by making the
smallest element follow the largest one. Conversely, we can cut a
cyclic order at any point to obtain a linear order.
\begin{defi}\label{Def_InducedOrders}
    Let $M$ be a finite set.
    \begin{enumerate}
        \item Let $m_1 < m_2 < \dots < m_{|M|}$ be a total order on $M$.
            The \textbf{induced cyclic order} is the cyclic order
            \begin{equation*}
                \sigma_<: M \to M,\quad x \mapsto 
                    \begin{cases}
                        m_{i+1} & x = m_i \text{ for } i < |M|, \\
                        m_1 & x = m_{|M|}.
                    \end{cases}
            \end{equation*}
        \item Let $\sigma$ be a cyclic order on $M$. Given $m \in M$, the 
            \textbf{induced linear order} $<_\sigma$ is defined as
            \begin{equation*}
                m <_\sigma \sigma(m) <_\sigma \sigma^2(m) <_\sigma \dots <_\sigma \sigma^{|M|-1}(m).
            \end{equation*}
    \end{enumerate}
\end{defi}
A partition can be intersection--free for both linear and cyclic orders.
\begin{defi}\label{Def_IntersectionFreedom}
    Let $M$ be a finite set and $\mathcal{P}$ a partition of 
    two--element--subsets of $M$. 
    \begin{enumerate}
        \item Let $<$ be a linear order on $M$. We call $\mathcal{P}$
            \textbf{intersection--free with respect to $<$} if 
            $\{m_1,m_2\}, \{n_1,n_2\} \in \mathcal{P}$ implies that
            $m_1 < n_1 < m_2 < n_2$ is impossible.
        \item Let $\sigma$ be a cyclic order on $M$. We call
            $\mathcal{P}$ \textbf{intersection--free with respect to $\sigma$}
            if there are no $m \in M$ and $1 \leq i < j < k < |M|$
            such that $\{m, \sigma^j(m)\} \in \mathcal{P}$ and
            $\{\sigma^i(m), \sigma^k(m)\} \in \mathcal{P}$.
    \end{enumerate}
\end{defi}
The conversion between linear and cyclic
orders does not change the intersections.
\begin{rem}\label{Rem_ConversionBetweenOrders}
    Let $M$ be a finite set and $\mathcal{P}$ a partition of
    two--element--subsets of $M$.
    \begin{enumerate}
        \item If $<$ is a linear order on $M$ and $\mathcal{P}$ is 
            intersection--free with respect to $<$, then 
            $\mathcal{P}$ is also intersection--free with respect
            to the induced cyclic order $\sigma_<$.
        \item If $\sigma$ is a cyclic order on $M$ and 
            $\mathcal{P}$ is intersection--free with respect to 
            $\sigma$, then $\mathcal{P}$ is also 
            intersection--free with respect to any induced linear
            order $<_{\sigma}$.
    \end{enumerate}
\end{rem}
\begin{proof}
    \begin{enumerate}
        \item Assume there are $1 \leq i < j < k$ and $m \in M$ such that
            $\{m,\sigma_<^j(m)\} \in \mathcal{P}$ and 
            $\{\sigma_<^i(m),\sigma_<^k(m)\} \in \mathcal{P}$.

            Let $n$ be the $<$--minimum of 
            $\{m, \sigma_<^i(m), \sigma_<^j(m), \sigma_<^k(m)\}$. Then
            $n = \sigma_<^s(m)$ with $s \in \{0,i,j,k\}$. We consider the
            case $s = i$ in detail, the other ones are analogous.

            We have 
            $\{m, \sigma_<^j(m)\} = \{ \sigma_<^{|M|-i}(n), \sigma_<^{j-i}(n) \}$ 
            and
            $\{\sigma_<^i(m),\sigma_<^k(m)\} = \{n, \sigma_<^{k-i}(n)\}$.
            Since 
            $n < \sigma_<^{j-i}(n) < \sigma_<^{k-i}(n) < \sigma_<^{|M|-i}(n)$,
            this contradicts $\mathcal{P}$ beeing intersection--free with
            respect to $<$.
        \item This is obvious. \qedhere
    \end{enumerate}
\end{proof}

\section{Circle representations}\label{Sec_CircleRepresentation}
We went to the effort of converting linear and cyclic orders into each
other to describe them as permutations.
For $n \in \N$ we use
$\bisn{n}$ to denote the set $\{1,2,\dots,n\}$.

Given a closed simplicial surface with $2n$ faces\footnote{In a closed 
simplicial surface, the number of faces is always even: By counting the 
number of edge--face--pairs in two different ways, we obtain that $2E = 3F$,
where $E$ is the number of edges and $F$ is the number of faces.} and 
an edge--3--colouring, each colour class defines a partition
of the faces into two--element--subsets (each edge is mapped to
the set of its two adjacent faces). This partition can also be
represented by a fix--point--free involution in $\Sym(\bisn{2n})$.

\begin{ex}\label{Ex_TorusColouring}
    Up to renaming the colours, the torus from Example \ref{Ex_TorusIncidence} 
    has exactly
    one vertex--3--colouring with colour classes $\{D\}$, $\{B,C\}$, 
    and $\{A\}$.

    This induces the edge--3--colouring with colour classes $\{a,b,c,d\}$,
    $\{e,g,i,k\}$, and $\{f,h,j,l\}$. Each colour class defines a
    partition of the faces, which can be interpreted as a fix--point--free
    involution.

    \begin{tabular}{ccc}
        \underline{colour class} & \underline{colour partition} & \underline{colour involution} \\
        $\qquad\{a,b,c,d\}\qquad$ & $\{ \{2,7\}, \{3,6\}, \{1,4\}, \{5,8\} \}$ & $\qquad(1,4)(2,7)(3,6)(5,8)\qquad$ \\
        $\{e,g,i,k\}$ & $\{ \{1,2\}, \{3,4\}, \{5,6\}, \{7,8\} \}$ & $(1,2)(3,4)(5,6)(7,8)$ \\
        $\{f,h,j,l\}$ & $\{ \{2,3\}, \{4,5\}, \{6,7\}, \{1,8\} \}$ & $(1,8)(2,3)(4,5)(6,7)$ 
    \end{tabular}
\end{ex}

\begin{defi}\label{Def_ColourPartitionInvolution}
    Let $(V,E,F,\prec)$ be a closed simplicial surface with edge--3--colouring
    $c_E$. For each colour $C$ the \textbf{colour partition} is
    \begin{equation}
        \left\{ \{ f \in F \mid e \prec f \} \mid e\in E, c_E(e) = C \right\}
    \end{equation}
    and the \textbf{colour involution} is a map $\rho: F \to F$ that 
    assigns to
    each face $f$ the unique face $g$ sharing an edge of colour $C$ with it.
\end{defi}

With Definition \ref{Def_IntersectionFreedom} and Definition 
\ref{Def_ColourPartitionInvolution} we can reformulate the definition
of triangle--folding.
\begin{rem}\label{Rem_TriangleFoldingByIntersectionFreedom}
    Let $(V,E,F,\prec)$ be a closed simplicial surface with 
    vertex--3--colouring $c_V$ and induced edge--3--colouring $c_E$. 
    A total order $<$ on $F$ is a triangle--folding if and only
    if all colour partitions of $c_E$ are intersection--free with
    respect to $<$.
\end{rem}

We want to restate the folding restrictions in terms of
a cyclic permutation on the faces and the fix--point--free involutions
defined by the edge--3--colouring. Let $\sigma \in \Sym(\bisn{2n})$ denote
a cyclic permutation of the $2n$ faces and suppose that 
$\rho \in \Sym(\bisn{2n})$ is an involution representing an edge
colour class.

We will show that $\rho$ is intersection--free
with respect to $\sigma$ if and only if
their product $\sigma\rho$ has exactly $n+1$ cycles. To prove this
claim, we translate the permutations into a geometric setting. We arrange
the elements of $\bisn{2n}$ into a circle, with the order defined by
$\sigma$. The involution is depicted by connecting two points in the
same orbit.

Consider the involutions $(1,4)(2,7)(3,6)(5,8)$ and
$(1,2)(3,4)(5,6)(7,8)$ from the previous example, together with
the cyclic order $(1,4,3,8,5,2,7,6)$. The first one is
in\-ter\-sec\-tion--free, the second one is not.
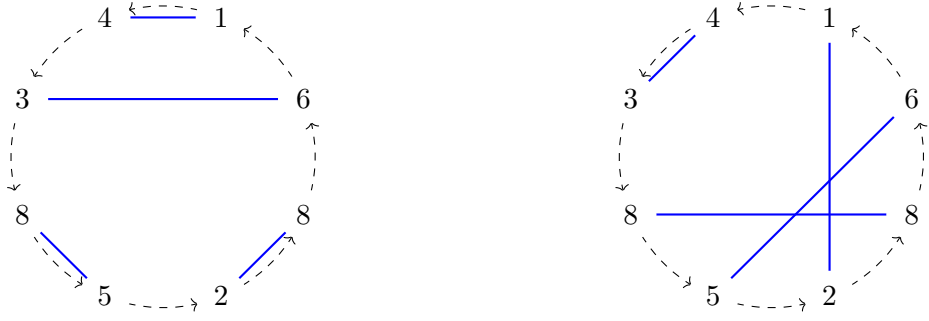
\newcommand{\circlePoints}{
    \def\off{9}
    \foreach \i in {1,...,8}{
        \coordinate (\i) at (-22.5+45*\i:1);
    }
}
\newcommand{\circleCircumference}{
    \def\off{9}
    \foreach \i in {1,...,8}{
        \draw[dashed,->] (-22.5+45*\i+\off:1) arc (-22.5+45*\i+\off:22.5+45*\i-\off:1);  
    }

}
\newcommand{\labelCirclePoints}{
    \foreach \i/\j in {1/6, 2/1, 3/4, 4/3, 5/8, 6/5, 7/2, 8/8}{
        \node[circle,fill=white] at (\i) {\j};
    }
}
\begin{center}
    \begin{tikzpicture}[scale=2]
        \begin{scope}
            \circlePoints
            \circleCircumference
            \foreach \i/\j in {2/3, 5/6, 1/4, 7/8}{
                \draw[\colE,thick] (\i) -- (\j);
            }
            \labelCirclePoints
        \end{scope}

        \begin{scope}[shift={(4,0)}]
            \circlePoints
            \circleCircumference
            \foreach \i/\j in {3/4, 7/2, 1/6, 5/8}{
                \draw[\colE,thick] (\i) -- (\j);
            }
            \labelCirclePoints
        \end{scope}
    \end{tikzpicture}
\end{center}

\begin{defi}\label{Def_CircleRepresentation}
    Let $n \in \N$ and $\sigma$ be a cyclic order on $\bisn{2n}$. Let
    $\{M_j\}_{1 \leq j \leq n}$ be a partition of $\bisn{2n}$ into
    two--element--subsets.
    A \textbf{circle representation} $\mathcal{C}$ 
    of $(\sigma, \{M_j\}_{1 \leq j \leq n})$ 
    is a triple 
    $(\iota, \{Z_k\}_{1\leq k\leq 2n}, \{S_j\}_{1\leq j \leq n})$ with
    \begin{enumerate}
        \item A map $\iota: \bisn{2n} \to \{ x \in \C \mid \norm{}{x} = 1 \}$
            with $\iota(\sigma(m)) = \e^{\frac{2\pi i}{2n}}\iota(m)$  for
            all $m \in \bisn{2n}$.
        \item The \textbf{arcs} 
            $Z_k \isdef \{ \iota(k) \cdot \e^{\frac{2\pi i}{2n}x} \mid x \in [0,1] \}$
            with \textbf{origin} $\iota(k)$ and \textbf{target} 
            $\iota(\sigma(k))$.
        \item The \textbf{line segments} 
            $S_j \isdef \{ \alpha\iota(x_1) + (1-\alpha)\iota(x_2) \mid M_j = \{x_1,x_2\}, \alpha \in [0,1] \}$.
    \end{enumerate}
    The circle representation is \textbf{intersection--free}
    if the line segments are pairwise disjoint.

    If $\{M_j\}_{1\leq j\leq n}$ is the set of orbits of a 
    fix--point--free involution $\rho$, we call $\mathcal{C}$
    a \textbf{circle representation of $(\sigma,\rho)$}.
\end{defi}

This notion of intersection--free is connected with the concepts
from Definition \ref{Def_IntersectionFreedom}.
\begin{lem}\label{Lem_IntersectionFreedomOfCircleAndCycle}
    Let $\sigma$ be a cyclic order on $\bisn{2n}$ and 
    $\{M_j\}_{1 \leq j \leq n}$ a partition of $\bisn{2n}$ into 
    two--element--subsets. Then the following statements are equivalent:
    \begin{enumerate}
        \item $\{M_j\}_{1 \leq j\leq n}$ is intersection--free with respect to $\sigma$.
        \item Any circle representation of $(\sigma, \{M_j\}_{1 \leq j \leq n})$
            is intersection--free.
    \end{enumerate}
\end{lem}
\begin{proof}
    Consider two line segments $\overline{\iota(a_1)\iota(a_2)}$ and
    $\overline{\iota(a_3)\iota(a_4)}$. Since rotations and reflections
    do not change the intersection of line segments, we may assume
    $\iota(a_1) = 1$. For $j \in \{1,2,3\}$ we write
    $\iota(a_j) = \e^{i\phi_j}$
    for some $0 < \phi_j < 2\pi$.

    $\{a_1,a_2\}$ and $\{a_3,a_4\}$ intersect with respect to $\sigma$ if
    and only if
    $0 < \phi_3 < \phi_2 < \phi_4$ or $0 < \phi_4 < \phi_2 < \phi_3$
    holds.
    
    The line segment $\overline{\iota(a_1)\iota(a_2)}$ divides the
    circle $\{ x \in \C | \norm{}{x} = 1 \}$ into two connected components.
    The two boundary components are parametrized by the 
    intervalls $[0,\phi_2]$ and 
    $[\phi_2,2\pi]$. We distinguish two cases:
    
    \begin{enumerate}
        \item If $\{\phi_3,\phi_4\} \cap [0,\phi_2]$ contains only one element,
            $\iota(a_3)$ and $\iota(a_4)$ lie in different connected 
            components. Since the circle is convex, the line segment 
            $\overline{\iota(a_3)\iota(a_4)}$
            lies within the circle, so it has to intersect the line
            segment $\overline{\iota(a_1)\iota(a_2)}$.
        \item If $\{\phi_3,\phi_4\} \cap [0,\phi_2]$ contains zero or two
            elements, $\iota(a_3)$ and $\iota(a_4)$ lie in the same connected
            component. Since both connected components are convex, the line
            segment $\overline{\iota(a_3)\iota(a_4)}$ 
            is completely contained in one of them. In particular, it
            does not intersect the line segment 
            $\overline{\iota(a_1)\iota(a_2)}$. \qedhere
    \end{enumerate}
\end{proof}

If we consider an
intersection--free circle representation of $(\sigma, \rho)$, 
we observe a correspondence
between bounded connected components and orbits of the product $\rho\sigma$.
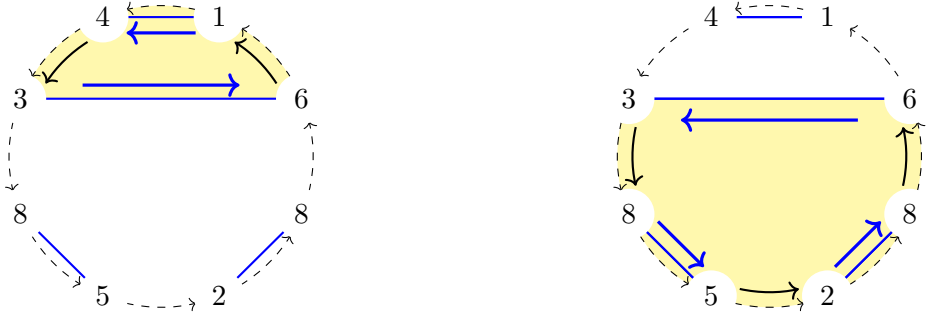
\begin{center}
    \begin{tikzpicture}[scale=2]
        \begin{scope}
            \circlePoints
            \begin{scope}
                \clip (0,0) circle (1);
                \fill[yellow!40!white] (1) -- ($(2)+(1)$) -- (2) -- ($(2)+(3)$) -- (3) -- ($(3)+(4)$) -- (4) -- cycle;
            \end{scope}

            \foreach \i/\j in {2/3, 5/6, 1/4, 7/8}{
                \draw[\colE,thick] (\i) -- (\j);
            }

	    \def\rad{0.9}
	    \def\winkel{10}
	    \foreach \s in {22.5,112.5}
            {
                    \draw[thick,->] ($({\rad*cos(\s+\winkel)},{\rad*sin(\s+\winkel)})$) arc(\s+\winkel:\s+45-\winkel:\rad);
            }
				
            \def\rad{0.85}
            \def\winkel{7}
            \draw[\colE,very thick,->] ($({\rad*cos(67.5+\winkel)},{\rad*sin(67.5+\winkel)})$) -- ($({\rad*cos(112.5-\winkel)},{\rad*sin(112.5-\winkel)})$);
            
            \def\rad{0.7}
            \def\winkel{20}
            \draw[\colE,very thick,->] ($({\rad*cos(157.5-\winkel)},{\rad*sin(157.5-\winkel)})$) -- ($({\rad*cos(22.5+\winkel)},{\rad*sin(22.5+\winkel)})$);
            
            \circleCircumference
            \labelCirclePoints
        \end{scope}

        \begin{scope}[shift={(4,0)}]
            \circlePoints

            \begin{scope}
                \clip (0,0) circle (1);
                \fill[yellow!40!white] (4) -- ($(4)+(5)$) -- (5) -- ($(5)+(6)$) -- (6) -- ($(6)+(7)$) -- (7) -- ($(7)+(8)$) -- (8) -- ($(8)+(1)$) -- (1) -- cycle;
            \end{scope}

            \foreach \i/\j in {2/3, 5/6, 1/4, 7/8}{
                \draw[\colE,thick] (\i) -- (\j);
            }

            \def\rad{0.9}
            \def\winkel{10}
            \foreach \s in {157.5,-112.5,-22.5}
            {
                    \draw[thick,->] ($({\rad*cos(\s+\winkel)},{\rad*sin(\s+\winkel)})$) arc(\s+\winkel:\s+45-\winkel:\rad);
            }
            
            \def\rad{0.85}
            \def\winkel{8}
            \foreach \s in {202.5,292.5}
                    \draw[\colE,very thick,->] ($({\rad*cos(\s+\winkel)},{\rad*sin(\s+\winkel)})$) -- ($({\rad*cos(\s+45-\winkel)},{\rad*sin(\s+45-\winkel)})$);
            
            \def\rad{0.63}
            \def\winkel{0}
            \draw[\colE,very thick,->] ($({\rad*cos(22.5+\winkel)},{\rad*sin(22.5+\winkel)})$) -- ($({\rad*cos(157.5-\winkel)},{\rad*sin(157.5-\winkel)})$);

            \foreach \i/\j in {2/3, 5/6, 1/4, 7/8}{
                \draw[\colE,thick] (\i) -- (\j);
            }
            \circleCircumference
            \labelCirclePoints
        \end{scope}
    \end{tikzpicture}
\end{center}
It connects a bounded connected component with the origins of the arcs
contained in its boundary, which are in turn given as the orbits of 
$\rho\sigma$. In the illustration above, the orbits are $\{1,3\}$ and
$\{2,6,8\}$.

To prove the correspondence, we have to analyse the bounded connected
components of a circle representation in detail. This will take the remainder
of this section and culminate in theorem
\ref{Theo_IntersectionFreedomByCycleCounting}.

\begin{lem}\label{Lem_NumberConnectedComponents}
    Let $(\iota, \{Z_k\}_{1\leq k\leq 2n}, \{S_j\}_{1\leq j\leq n})$
    be an intersection--free circle representation and 
    $U \isdef \bigcup_{k = 1}^{2n} Z_k \cup \bigcup_{j = 1}^{n} S_j$.
    Then $\C \ohne U$
    has exactly $n+1$ bounded connected components.
\end{lem}
\begin{proof}
    We interpret the sets $Z_k$ and $S_j$ as the edges of a planar graph
    with vertices
    $\{\iota(k) | 1 \leq k \leq 2n\}$.
    This graph has $2n$ vertices and $2n+n$ edges. Since its 
    \texttt{Euler}--cha\-rac\-te\-ris\-tic is 1,
    there are exactly $n+1$ bounded connected components.
\end{proof}

The main technical analysis is contained in the following lemma, which 
describes
the specific form of the bounded connected components explicitly.
\begin{lem}\label{Lem_ConnectedComponentCharacterisation}
    Let $(\iota, \{Z_k\}_{1\leq k\leq 2n}, \{S_j\}_{1\leq j\leq n})$
    be an intersection--free circle representation and 
    $U \isdef \bigcup_{k = 1}^{2n} Z_k \cup \bigcup_{j = 1}^{n} S_j$.
    The connected components of the complement 
    $\C \ohne U$ have the properties:
    \begin{itemize}
        \item There is exactly one unbounded connected component
            $\{x \in \C  | \norm{}{x} > 1 \}$. Its boundary is
            $\bigcup_{k = 1}^{2n}Z_j$.
        \item There are exactly $n+1$ bounded connected components. 
            The
            boundary of each bounded connected component has the form
            $\bigcup_{t = 1}^m (Z_{k_t} \union S_{j_t})$ for some $m \in \N$
            such that
            \begin{enumerate}
                \item The intersections $Z_{k_t} \cap S_{j_t}$ and 
                    $S_{j_t} \cap Z_{k_{t+1}}$ contain exactly one element
                    for each $1 \leq t \leq m$ (we set $j_{m+1} \isdef j_1$
                    for this purpose).
                \item If $\iota(a_t)$ is the unique element in 
                    $S_{j_{t-1}} \cap Z_{k_t}$ and $\iota(b_t)$ is the unique
                    element in $Z_{k_t} \cap S_{j_t}$, then
                    $\sigma(a_t) = b_t$.
            \end{enumerate}
    \end{itemize}
\end{lem}
This lemma characterises the boundary of the bounded connected components.
The first condition states that each boundary consists of an alternating
sequence of complete arcs and line segments. The second one defines the
order in which this sequence should be traversed.
\begin{proof}
    Since all line segments lie in $\{ x \in \C | \norm{}{x} \leq 1 \}$ and
    $\bigcup_{j=1}^{2n} Z_j = \{ x \in \C | \norm{}{x}=1 \}$,
    the claim concerning the unbounded connected component is clear.

    Let $\kappa$ be a bounded connected component. Since the circle 
    representation is intersection--free, the boundary of $\kappa$ is the
    union of some $Z_k$ and $S_j$.
    
    \begin{enumerate}
        \item
            Let $\iota(k)$ be in the boundary of $\kappa$. This lies in $Z_k$,
            $Z_{k+1}$ and a line segment. 
            Since the line segment divides the circle
            $\{x \in \C \mid \norm{}{x} \leq 1\}$ into two connected components,
            it lies in the boundary of exactly two connected components (since they
            are
            intersection--free). As
            $\iota(k)$ is an endpoint for the segment, any connected component
            which
            has the line segment as a boundary will also have one of the arcs as 
            boundary. Since the two connected components adjacent to the 
            segment contain different arcs, $Z_k$ and $Z_{k+1}$ cannot be
            both contained in $\kappa$.

            Therefore, for each point $\iota(k)$ in the boundary of $\kappa$,
            there is exactly one arc and exactly one line segment in the
            boundary. It follows that arcs and line segments alternate
            along the boundary of $\kappa$, which shows the first property.
        \item We have either $\sigma(a_1) = b_1$ or $\sigma(b_1) = a_1$. If
            the second equality holds, we invert the ordering of the sets in the 
            boundary. Hence, suppose $\sigma(a_1) = b_1$.

            Then $\iota(b_1)$ is the target of an arc. The other point in
            $S_{j_1}$ is $\iota(a_2)$. This point is the source of the
            arc $Z_{k_2}$ (otherwise the two arcs would be separated by
            the line segment). In particular, we have $\sigma(a_2) = b_2$.
            The full claim follows by induction. \qedhere
    \end{enumerate}
\end{proof}

This technical characterisation allows the construction of the bijection
between connected components and orbits of a group.
\begin{lem}\label{Lem_BijectionConnectedComponentsOrbits}
    Let $\sigma$ be a cyclic order on $\bisn{2n}$ and $\rho$ a fix--point--free
    involution on $\bisn{2n}$. Let 
    $(\iota, \{Z_k\}_{1\leq k\leq 2n}, \{S_j\}_{1\leq j \leq n})$ be an
    intersection--free circle representation of $(\sigma,\rho)$ and use
    the notation from Lemma \ref{Lem_ConnectedComponentCharacterisation}.
    
    Then we have a bijection between the bounded connected components of
    $\C \ohne U$ and the orbits of $\rho\sigma$: a connected
    component with boundary 
    $\bigcup_{t = 1}^m (Z_{k_t} \union S_{j_t})$
    is mapped to
    $\bigcup_{t = 1}^m (S_{j_{t-1}} \cap Z_{k_t})$ (with $j_{0} \isdef j_m$),
    which consists of the origins of the arcs contained in the
    boundary of the connected component.
\end{lem}
\begin{proof}
    The map is well--defined: 
    Let $\iota(a_t)$ be the
    unique element of $S_{k_{t-1}} \cap Z_{k_t}$, then $\sigma(\iota(a_t))$
    is the unique element in $Z_{k_t} \cap S_{j_t}$. In particular, it
    lies in $S_{j_t}$. Then $\rho\sigma(\iota(a_t))$ is the other endpoint of
    this line segment, and therefore the unique element in the
    intersection $S_{j_t} \cap Z_{k_{t+1}}$. Therefore, this image
    is an orbit of $\rho\sigma$.

    To show injectivity, we consider an element $\iota(t)$. This is the
    source of one arc and the target of another. Since only sources are
    mapped by our construction, the images of different bounded 
    connected components are disjoint.

    For surjectivity, consider an orbit $B$ of $\rho\sigma$. This orbit
    is not empty, so there is a $b \in B$. This $b$ is the source of
    one arc $Z$. Then the unique bounded connected component with $Z$
    in its boundary is mapped to $B$.
\end{proof}
By combining Lemma \ref{Lem_BijectionConnectedComponentsOrbits} and
Lemma \ref{Lem_NumberConnectedComponents}, we have shown that
a circle representation of $(\sigma,\rho)$ being intersection--free implies that
$\rho\sigma$ has exactly $n+1$ cycles.

We now proceed in the opposite direction: If $\sigma\rho$ has exactly $n+1$ 
cycles, we construct an intersection--free circle representation of 
$(\sigma,\rho)$.

To do so, we need some general properties of permutations.
\begin{lem}\label{Lem_PermutationFixedPoint}
    Let $\pi \in \Sym(\bisn{n})$ be a permutation with more than $\frac{n}{2}$ cycles.
    Then $\pi$ has at least one fixed point.
\end{lem}
\begin{proof}
    Suppose $\pi$ has no fixed point. Then every orbit of $\langle\pi\rangle$
    contains at least two 
    elements. Since orbits are disjoint, $\bisn{n}$ has to contain more than
    $\frac{n}{2} \cdot 2$ elements, which is a contradiction.
\end{proof}

If $\rho\sigma$ has exactly $n+1$ cycles, it needs to have a fixed point.
This is significant: In a circle representation, a fixed point 
corresponds to a line
segment whose end--points have minimal distance. In particular, it can
be removed from a circle representation without changing whether it is
intersection--free.

To prove this reduction, we first show that after the removal of this
line segment there is another fixed point. Then we use an inductive
argument to prove the complete claim.

\begin{lem}\label{Lem_PermutationReduct}
    Let $\sigma$ be a $2n$--cycle and $\rho$ a fix--point--free 
    involution in $\Sym(\bisn{2n})$,
    such that $\rho\sigma$ has exactly $n+1$ cycles ($n > 1$).
    Let $f \in \bisn{2n}$ be a fixed point of $\rho\sigma$.

    We define permutations on $\bisn{2n}\ohne\{f,\sigma(f)\}$ 
    (remember that $\{f,\sigma(f)\}$ is an orbit of $\rho$):
    \begin{align*}
        \hat{\rho}(k) &\isdef \rho(k)
        &
        \hat{\sigma}(k) &\isdef 
            \begin{cases}
                \sigma^3(k) & \sigma(k) = f \\
                \sigma(k) & \text{otherwise}
            \end{cases}
    \end{align*}

    Then $\hat{\rho}\hat{\sigma}$ has exactly $n$ cycles.
\end{lem}
\begin{proof}
    To show the claim, we analyse how the orbits $B$ of $\rho\sigma$ change.
    \begin{enumerate}
        \item \underline{$B = \{f\}$}: This orbit is removed.
        \item \underline{$\sigma(f) \in B$}: The only element on which the
            orbit might change is the precursor of $\sigma(f)$, namely
            $(\rho\sigma)^{-1}\sigma(f) \in B$. Since $\rho = \rho^{-1}$, we
            have
            \begin{equation*}
                (\rho\sigma)^{-1}\sigma(f) = \sigma^{-1}\rho^{-1}\sigma(f) = \sigma^{-1}(\rho\sigma)(f) = \sigma^{-1}(f)
            \end{equation*}
            Now we compute its image under the modified permutations:
            \begin{equation*}
                \hat{\rho}\hat{\sigma}\sigma^{-1}(f) = \rho\sigma^3\sigma^{-1}(f) = (\rho\sigma)(\sigma(f))
            \end{equation*}
            In other words, the image of $\sigma^{-1}(f)$ under 
            $\hat{\rho}\hat{\sigma}$
            is the same as the image of $\sigma(f)$ under
            $\rho\sigma$. The new orbit is $B\ohne\{\sigma(f)\}$.

            We have to show that this new orbit is not empty. For it
            to be empty, we would need $\sigma^2(f) = f$, i.\,e. $n=1$.
        \item \underline{$f,\sigma(f) \not\in B$}: In this case, nothing 
            changes. The only difference can appear if there is an element $k$ 
            of the orbit such that $\sigma(k) = f$. In this case 
            $\rho\sigma(k) = \sigma(f)$.
    \end{enumerate}
    In total, the number of orbits is reduced by 1.
\end{proof}

We now proceed with formulating the induction.
\begin{lem}\label{Lem_FixedPointInduction}
    Let $\sigma$ be a $2n$--cycle and $\rho$ a fix--point--free involution
    in $\Sym(\bisn{2n})$,
    such that $\rho\sigma$ has exactly $n+1$ cycles. Then every circle
    representation of $(\sigma,\rho)$ 
    is intersection--free.
\end{lem}
\begin{proof}
    We prove the claim by induction on $n$. For $n=1$, we have 
    $\sigma = \rho = (1,2)$. Therefore, there is only one line segment,
    which cannot intersect with any other one.

    Now assume $n > 1$. By Lemma \ref{Lem_PermutationFixedPoint} 
    there is a fixed point $f$ of $\rho\sigma$,
    i.\,e. $\rho(\sigma(f)) = f$. Denote $g \isdef \sigma(f)$, then we have
    $\rho(g) = f$.

    The line segment between $f$ and $g$ cannot intersect any other
    line segment, since $f$ and $g$ are neighbours on the circle. We only
    need to show that the other line segments do not intersect.

    To do so, we define $\hat{\rho}$ and $\hat{\sigma}$ as in Lemma
    \ref{Lem_PermutationReduct}.
    Their product has exactly $n$ cycles. By the
    induction hypothesis, their circle representation is 
    intersection--free. 
\end{proof}

We can now formulate the main theorem.

\begin{theo}\label{Theo_IntersectionFreedomByCycleCounting}
    Let $\sigma$ be a $2n$--cycle and $\rho$ a fix--point--free 
    involution on $\bisn{2n}$. Then the following are equivalent:
    \begin{itemize}
        \item The orbits of $\rho$ are intersection--free with
            regards to $\sigma$.
        \item $\rho\sigma$ has exactly $n+1$ orbits on $\bisn{2n}$.
    \end{itemize}
\end{theo}
\begin{proof}
    By Lemma \ref{Lem_IntersectionFreedomOfCircleAndCycle}, the
    orbits of $\rho$ are intersection--free if and only if a circle
    representation is intersection--free. Combining Lemma
    \ref{Lem_NumberConnectedComponents} and Lemma
    \ref{Lem_BijectionConnectedComponentsOrbits} then gives one
    direction of the proof. The other direction follows from
    Lemma \ref{Lem_FixedPointInduction}.
\end{proof}

\begin{cor}\label{Cor_FoldingOntoTriangle}
    Let $(V,E,F,\prec)$ be a closed simplicial surface with
    vertex--3--colouring $c_V$ and induced edge--3--colouring $c_E$. 

    There is a triangle--folding of the surface if and only if there
    exists an $|F|$--cycle $\sigma \in \Sym(F)$ such that its product
    with all colour involutions has exactly $\frac{|F|}{2}+1$ cycles.
\end{cor}
\begin{proof}
    By Remark \ref{Rem_TriangleFoldingByIntersectionFreedom}, $<$ is a
    triangle--folding if and only if all colour partitions are
    intersection--free with respect to $<$. By remark
    \ref{Rem_ConversionBetweenOrders} we can replace $<$ by a cyclic
    order $\sigma$. Since the colour partitions are the orbits
    of the colour involutions, the claim follows from theorem
    \ref{Theo_IntersectionFreedomByCycleCounting}.
\end{proof}

\section{Orientability}\label{Sec_Orientation}
In this section, we show that only orientable simplicial surfaces can
be folded onto a triangle. 
This is based on the following observation:
\begin{rem}\label{Rem_DistanceInCircleRepresentation}
    Let $\overline{xy}$ be a line segment in an intersection--free
    circle representation. Then $y = \e^{2\pi i \frac{k}{2n}}x$
    with $k$ odd.
\end{rem}
\begin{proof}
    By the definition of circle representations, there is a $k \in \Z$
    with 
    $y = \e^{2\pi i \frac{k}{2n}}x$.

    This line segment divides the circle into two connected components.
    The boundary of one of them consists precisely of the points
    $\{ \e^{2\pi i\frac{l}{2n}} | 0 \leq l \leq k \}$.
    These points have to be connected in pairs of two, therefore the
    number of elements in this set has to be even. This is only possible
    if $k$ is odd.
\end{proof}

Since a circle representation subdivides the circle into an even number
of points, the concept of even and odd distances makes sense.
\begin{defi}\label{Def_EvenOddDistance}
    Let $\mathcal{C}$ be a circle representation with map 
    $\iota: \bisn{2n} \to \C$. For two
    elements $a,b \in \bisn{2n}$ there is a $k \in \Z$ with
    \begin{equation*}
        \iota(y) = \e^{2\pi i\frac{k}{2n}}\iota(x).
    \end{equation*}
    If $k$ is even, $a$ and $b$ have \textbf{even distance}. Otherwise,
    they have \textbf{odd distance}.
\end{defi}

\begin{rem}\label{Rem_EvenDistanceEquivalence}
    All points with even distance from each other form an
    equivalence class.
\end{rem}
Combining Remark \ref{Rem_DistanceInCircleRepresentation} and Remark
\ref{Rem_EvenDistanceEquivalence}, we get a necessary 
criterion for foldability:
\begin{lem}\label{Lem_EvenDistanceClassByInvolutions}
    Let $\rho_1,\rho_2,\rho_3$ be three fix--point--free involutions
    on $\bisn{2n}$ such that there is a cyclic order $\sigma$
    that is intersection--free with respect to the orbit partition of 
    each of them.
    
    If $\langle \rho_1,\rho_2,\rho_3\rangle$ is transitive on
    $\bisn{2n}$, the even--distance equivalence relation is
    only dependent on the involutions.
\end{lem}
\begin{proof}
    By Remark \ref{Rem_DistanceInCircleRepresentation}, 
    two numbers in an orbit of any $\rho_i$ have
    to lie in different equivalence classes. Since there are
    only two classes and $\langle \rho_1,\rho_2,\rho_3\rangle$ is
    transitive on $\bisn{2n}$,
    this determines the class membership of every number.
\end{proof}

A more pedestrian formulation would be, whether it is possible
to relabel $\bisn{2n}$ in such a way, that all involutions always
swap even and odd numbers.

Our next step is the geometric interpretation of this equivalence
relation. An orientation maps each face to a cyclic order of its
vertices such that the orders of adjacent faces are compatible. For
convenience, we only define this for simplicial surfaces with
vertex--3--colourings.
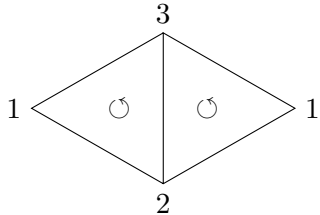
\begin{center}
    \begin{tikzpicture}[scale=2]
        \coordinate[label=below:2] (D) at (0,0);
        \coordinate[label=left:1] (L) at (150:1);
        \coordinate[label=above:3] (U) at (90:1);
        \coordinate[label=right:1] (R) at (30:1);

        \draw (U) -- (L) -- (D) -- (R) -- (U) -- (D);
        \node at (barycentric cs:D=1,L=1,U=1) {$\circlearrowleft$};
        \node at (barycentric cs:D=1,R=1,U=1) {$\circlearrowleft$};
    \end{tikzpicture}
\end{center}
\begin{defi}\label{Def_SimplicialOrientation}
    Let $(V,E,F,\prec)$ be a simplicial surface with vertex--3--colouring
    $c_V$. A \textbf{simplicial orientation} is a map
    $z: F \to \{ (1,2,3), (1,3,2) \}$, such that, for two faces $f_1$ and
    $f_2$ with a common edge, $z(f_1) = z(f_2)^{-1}$ holds.
\end{defi}

\begin{theo}\label{Theo_FoldingImpliesOrientation}
    Let $S$ be a simplicial surface that can be folded onto a 
    triangle. Then $S$ has a simplicial orientation.
\end{theo}
\begin{proof}
    By Remark \ref{Rem_TriangleCoveringClassification}, 
    $S$ has a vertex--3--colouring and an induced
    edge--3--colouring with involutions $\rho_1$, $\rho_2$,
    and $\rho_3$. We argue for each orbit of 
    $\langle\rho_1,\rho_2,\rho_3\rangle$ separately.

    By Remark \ref{Rem_EvenDistanceEquivalence}, 
    the faces fall in two equivalence classes. Mapping one of the
    classes to $(1,2,3)$ and the other one to $(1,3,2)$ defines
    a simplicial orientation.
\end{proof}

\section{Acknowledgements}
I am grateful for the support of my supervisor, Alice Niemeyer, in 
writing this paper. I am also grateful for the funding by 
the ``Graduiertenkolleg Experimentelle konstruktive Algebra'' 
during the writing of this paper.

\cleardoublepage
\addcontentsline{toc}{section}{References}
\newcommand{\etalchar}[1]{$^{#1}$}
\providecommand{\bysame}{\leavevmode\hbox to3em{\hrulefill}\thinspace}
\providecommand{\MR}{\relax\ifhmode\unskip\space\fi MR }
\providecommand{\MRhref}[2]{%
  \href{http://www.ams.org/mathscinet-getitem?mr=#1}{#2}
}
\providecommand{\href}[2]{#2}


\end{document}